\documentclass[10pt]{amsart}

\usepackage{geometry,graphicx,amssymb,amsmath,amsbsy,eucal,amsfonts,mathrsfs,amscd,bm}
\usepackage[all]{xy}
\usepackage{tikz-cd}
\usetikzlibrary{arrows,shapes,calc,snakes}
\usetikzlibrary{shapes,snakes}
\usepackage{tikz-3dplot}
\usetikzlibrary{intersections}
\usepackage{tkz-euclide}
\usetkzobj{all}
\usetikzlibrary{angles}
\usepackage{pgfplots}
\usepackage{eufrak}
\usepackage[foot]{amsaddr}

\numberwithin{equation}{section}

\allowdisplaybreaks[3]

\newtheorem{theorem}{Theorem}[section]
\newtheorem{lemma}[theorem]{Lemma}

\newtheorem{corollary}[theorem]{Corollary}
\newtheorem{proposition}[theorem]{Proposition}
\theoremstyle{definition}
\newtheorem{definition}[theorem]{Definition}
\theoremstyle{remark}
\newtheorem{remark}[theorem]{Remark}

\newcommand{\rot}{\text{\rm rot }}
\newcommand{\rote}{\text{\rm rot}}

\newcommand{\p}{{\partial}}

\newcommand{\nab}{\nabla}

\newcommand{\mct}{\mathcal{T}_h}
\newcommand{\mctps}{\mathcal{T}_h^{\text{ps}}}

\newcommand{\dive}{{\ensuremath\mathop{\mathrm{div}\,}}}
\newcommand{\Div}{{\rm div}\,}

\newcommand{\jmp}[1]{ [\![ {#1}  ]\!] }

\newcommand{\pol}{\EuScript{P}}

\newcommand{\bbR}{\mathbb{R}}

\newcommand{\VV}{\mathcal{V}}

\newcommand{\bbN}{\mathbb{N}}
\newcommand{\lrarrow}{\hspace{5pt}\longrightarrow\hspace{5pt}}
\newcommand{\range}{{\rm range}\,}
\newcommand{\bbV}{\mathbb{V}}

\newcommand{\bbT}{\mathbb{T}}
\newcommand{\bbE}{\mathbb{E}}

\newcommand{\PS}{{\rm PS}}

\newcommand{\PST}{T^{{\rm ps}}}

\title[Powell-Sabin]{Exact sequences on Powell-Sabin splits}

\author{J.~Guzm\'an,  A. Lischke, M.~Neilan}

\begin{document}

\maketitle

\begin{abstract}
We construct {smooth} finite elements spaces on Powell-Sabin triangulations that form 
an exact sequence. 
{The first space of the sequence coincides
with the classical $C^1$ Powell--Sabin space,
while the others form stable and divergence--free yielding
pairs for the Stokes problem.}
We develop degrees of freedom for these spaces
that induce  projections that commute with {the differential operators}.
\end{abstract}

\thispagestyle{empty}
\section{Introduction}
In the finite element exterior calculus \cite{AFW06, AFW10},
sequences of discrete spaces that conform to the continuous de Rham complex
are used to approximate solutions of the Hodge--Laplacian.
While this framework has been successfully applied to the de Rham complex with minimal $L^2$ smoothness, recent progress has extended this methodology to higher order Sobolev spaces, i.e., spaces with greater smoothness. 
Such constructions naturally lead to structure--preserving discretizations for 
the Stokes/Navier--Stokes problem as well as problems in linear elasticity.
For example, in recent work \cite{FuGuzmanNeilan19, ChristiansenHu} 
 specific mesh refinements were used
to build spaces of continuous piecewise polynomial $k$-forms
with continuous exterior derivative.
In particular,  it is shown in \cite{FuGuzmanNeilan19} that {\em locally}, smooth finite element spaces form an exact sequence on so--called 
Alfeld splits in any spatial dimension and for any polynomial degree.  Global spaces in three dimensions
are also constructed in \cite{FuGuzmanNeilan19}, leading to stable finite element pairs
for the Stokes problem (also see \cite{Zhang04}). On the other hand, Christiansen and Hu \cite{ChristiansenHu} considered low-order approximations in any dimension. However they use different splits as they move along the de Rham sequence. For zero forms they use the finest split (e.g. in two dimensions it is the Powell-Sabin split). For $n-1$ forms, where $n$ is the dimension, they use the Alfeld split.

In this paper we construct smooth finite element spaces 
 on Powell--Sabin splits that form an exact sequence.  
In the lowest order case, the first space in the sequences coincides with the piecewise quadratic $C^1$ Powell--Sabin space \cite{PowellSabin77,SplineBook}. However, we construct these spaces for any polynomial degree which appears to be new (cf.~\cite{Gros16A,Gros16B}). We also define smooth spaces on Powell-Sabin splits for vector-valued polynomial spaces,
define commuting projections onto the finite element spaces,
and characterize the range and kernel of differential operators acting on the finite element spaces.
The last two spaces in the sequence form stable finite element pairs
for the Stokes problem that enforce the incompressibility constraint exactly; see \cite{JohnEtal}.

 A potential advantage of the use of Powell-Sabin splits is that
 the minimal polynomial degree of the global spaces
 is not expected to increase with respect to the spacial dimension.
For example, the lowest polynomial degree of $C^1$ spaces on Powell--Sabin splits
is two in both two and three dimensions.  In contrast, the polynomial degree of smooth
piecewise polynomials must necessarily increase with dimension on Alfeld splits.
In two dimensions, $C^1$ piecewise polynomials have degree of at least three,
 whereas in three dimensions the minimal polynomial degree is five \cite{Alfeld84,SplineBook}.
 These degree restrictions for $C^1$ conforming spaces also dictate 
the polynomial degrees of other finite element spaces on Alfeld splits.  For example, 
finite element spaces that approximate the velocity
in the Stokes problem must have degree of at least the spatial dimension \cite{ArnoldQin92,Zhang04,GuzmanNeilan18}.

Let us describe the Powell-Sabin split here. Let $\Omega \subset \mathbb{R}^2$ be a polyhedral domain, and let $\mct$ be a simplicial,
shape--regular triangulation of $\Omega$. Then the Powell-Sabin  triangulation $\mctps$ is obtained as follows. We select an interior point of each triangle $T \in \mct$ and adjoin this point with each vertex of $T$. Next, the interior points of each adjacent pair of triangles are connected with an edge. For any $T$ that shares an edge with the boundary of $\Omega$, an arbitrary point on the boundary edge is selected to connect with the interior point of $T$, so that each $T \in \mct$ is split into six triangles. See Figure \ref{fig:GlobalMesh}. In order for the resulting refinement $\mctps$ to be well-defined, the interior points must be selected such that their adjoining edge intersects the edge shared by their respective triangles in $\mct$, in which case $\mctps$ is the Powell-Sabin refinement of $\mct$. One common choice of interior points that produces a well-defined triangulation is the incenter of each $T \in \mct$, i.e., the center point of the largest circle that fits within $T$ \cite{SplineBook}. We define the set $\mathcal{M}(\mctps)$ to be the points of intersection of the edges of $\mct$ with the edges that adjoin interior points.  An interesting fact about the meshes constructed is that the points in $\mathcal{M}(\mctps)$ are singular vertices of the mesh $\mctps$; see  \cite{ScottVogelius85}. Hence, the last space in our sequence has to be modified accordingly; see the global space $\VV_{r}^2(\mctps)$ below. 

Related to the current work
is \cite{Zhang08,Zhang11},
where conforming finite element pairs
are proposed and studied for the Stokes problem
on Powell--Sabin meshes.
There it is shown that if the discrete velocity
space is the linear Lagrange finite element space, and if the pressure
space is the image of the divergence operator
acting on the discrete velocity space,
then the resulting pair is inf--sup stable.
Note
that, by design, the discrete pressure spaces in \cite{Zhang08,Zhang11}, and correspondingly
the range of the divergence operator, is not explicitly given.
Practically, this issue is bypassed by using the iterative penalty method
to solve the finite element method without explicitly constructing a basis
of the discrete pressure space.  In this paper we explicitly construct the discrete pressure space and characterize
the space of divergence--free functions for any polynomial degree.


The rest of the paper is organized as follows.  In the next section we state some preliminary 
definitions and results on a single macro-triangle.  In Section \ref{sec:SequencesMacro}
we show that the smooth finite element spaces form an exact sequence
on macro-triangles, and in Section \ref{sec:commuting_projections} we develop degrees of freedom and projections for these spaces,
and prove commutative properties of these projections.
We extend these results to the global setting in Section \ref{sec:Global}
and derive similar results.  We end the paper in Section \ref{sec:Conclusions}
with some concluding remarks.

\begin{center}
\begin{figure}[ht!]
\includegraphics[width=.49\textwidth]{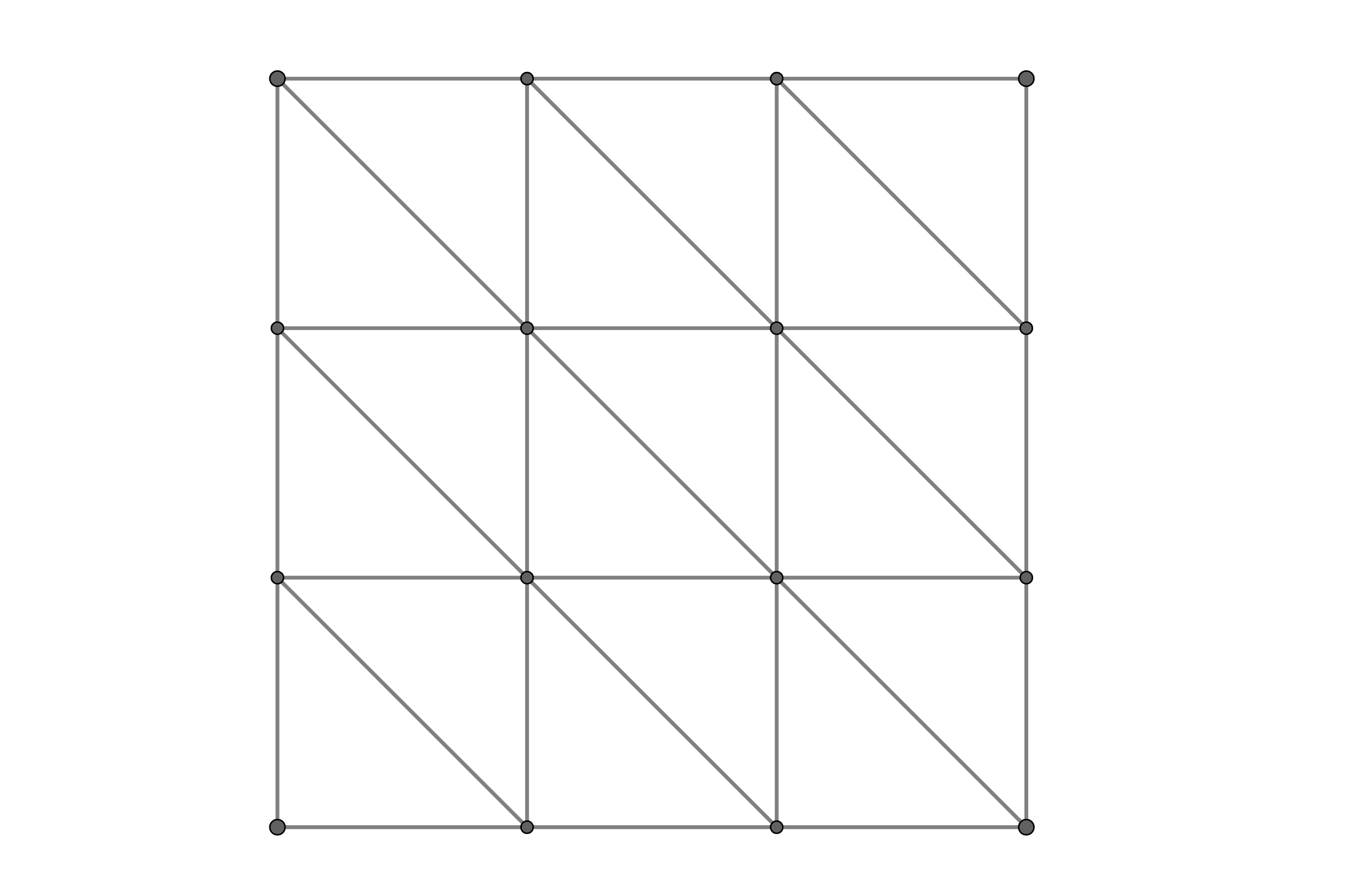}
\includegraphics[width=.49\textwidth]{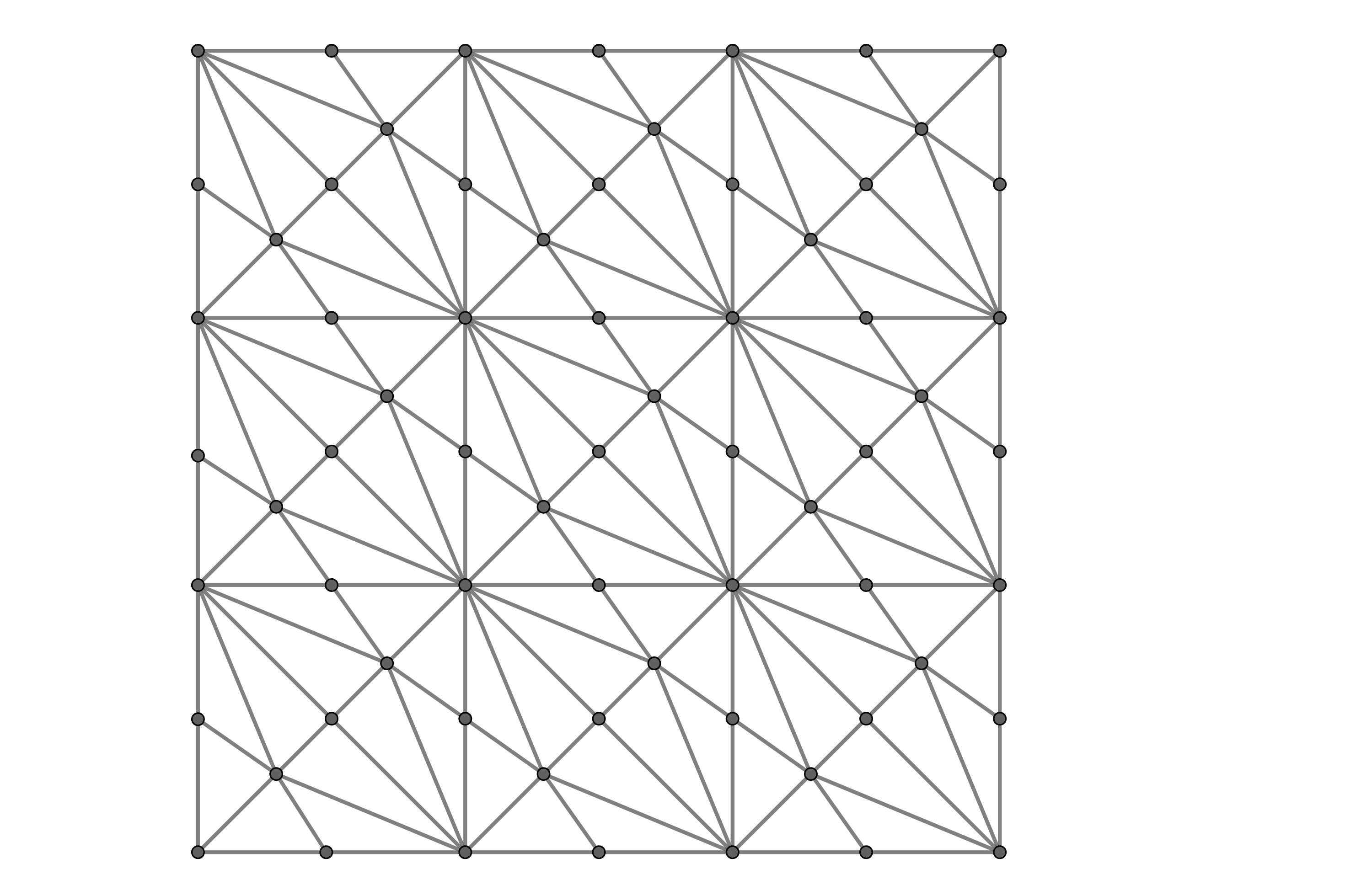}
\caption{(\emph{left}) A triangulation of the unit square, and (\emph{right}) its Powell-Sabin refinement.}
\label{fig:GlobalMesh}
\end{figure}
\end{center}

\section{Spaces on one macro-triangle}\label{sec:SpacesMacro}
Let $T$ be a triangle with vertices $z_1, z_2,$ and $z_3$, labelled counter--clockwise, and let $z_0$ be an interior point of $T$.
Denote the edges of $T$  by $\{e_i\}_{i=1}^3$, labelled such that $z_i$ is not a vertex of $e_i$, i.e., $e_i=[z_{i+1},z_{i+2}]$.
We denote
the outward unit normal of $\p T$ restricted to $e_i$
as $n_i$ and the tangent vector by $t_i$. Let $z_{3+i}$ be an interior point of edge $e_i$. We then construct the triangulation $\PST = \{T_1, \ldots, T_6 \}$ by connecting each $z_i$ to $z_0$ for $1 \le i \le 6$; see Figure \ref{fig:Notation}. We let $\mathcal{E}^b(\PST)$ be the set containing the six boundary edges of $\PST$. We also let $\mathcal{M}(\PST)=\{z_4, z_5, z_6\}$ and use the notation for $z \in \mathcal{M}(\PST)$, $\mathcal{T}(z)= \{K_1,  K_2\}$, where $K_i \in \PST$  have $z$ as a vertex. We also set $T(z)=K_1\cup K_2$.
Let $z \in \mathcal{M}(\PST)$ and suppose that $\mathcal{T}(z)=\{K_1, K_2\}$ with common edge $e$ then we define the jump as follows
\begin{equation*}
\jmp{p}(z)=  p_1(z) m_1+ p_2(z) m_2,
\end{equation*}
where $p_i=p|_{K_i}$ and $m_i$ is the outward pointing normal to $K_i$ perpendicular to $e$. We see then that $\jmp{p}(z)=  (p_1(z)-p_2(z)) m_1= -(p_1(z)-p_2(z)) m_2$.

\begin{center}
\begin{figure}
\includegraphics[width=.6\textwidth]{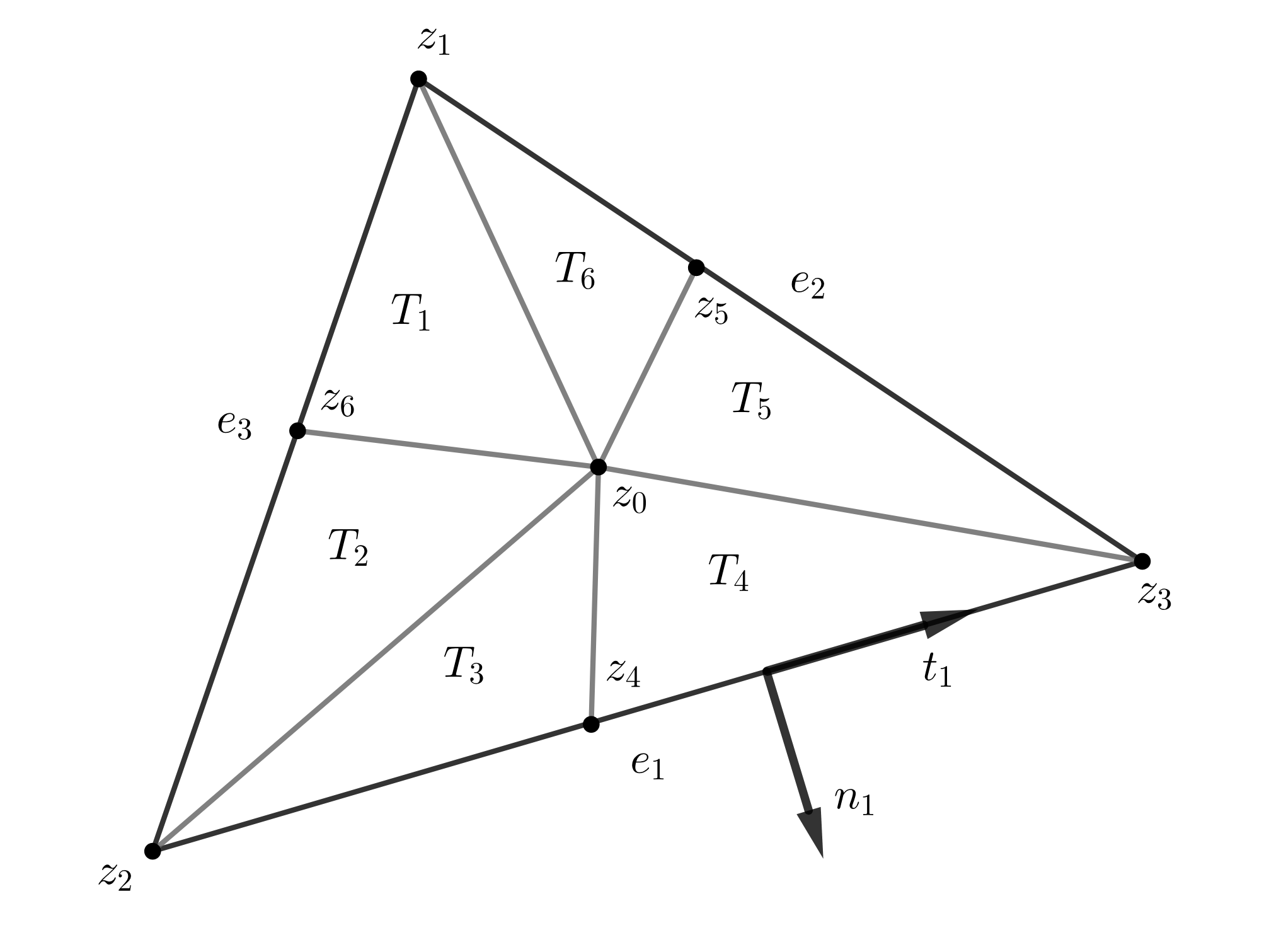}
\caption{A pictorial description of a Powell-Sabin split of a triangle.}
\label{fig:Notation}
\end{figure}
\end{center}

Let $\mu$ be the unique piecewise linear function on the mesh $\PST$ such that $\mu(z_0)=1$ and $\mu =0 $ on $\partial T$. 
We use the notation $\nab \mu_i:=\nab \mu |_{e_i} = \nab \mu|_{T(z_{3+i})}$ and note that
\begin{equation}\label{mun}
\frac{1}{| \nab \mu_i|} \nab \mu_i=-n_i\quad (i=1,2,3),
\end{equation}
and hence
\begin{equation}\label{mut}
\nab \mu_i \cdot t_i=0\quad (i=1,2,3).
\end{equation}

\subsection{Local finite element spaces}
In this section we consider three classes of finite element spaces each with varying smoothness on $\PST$. 
First we define the differential operators
\begin{align*}
\rot q = \big(\frac{\p q}{\p x_2},-\frac{\p q}{\p x_1}\big)^\intercal,\qquad \Div v = \frac{\p v_1}{\p x_1}+ \frac{\p v_2}{\p x_2},
\end{align*}
and corresponding spaces, for an open bounded domain $S\subset \bbR^2$,
\begin{alignat*}{2}
&H({\rm rot};S) = \{q\in L^2(S):\ \rot q\in L^2(S)\},\quad &&H({\rm div};S) = \{v\in [L^2({S})]^2:\ \Div v\in L^2({S})\},\\
&\mathring{H}({\rm rot};S) = \{q\in H({\rm rot};S):\ q|_{\p S} =0\},\quad &&\mathring{H}({\rm div};S) = \{v\in H({\rm div};S):\ v\cdot n_S|_{\p S}=0\},
\end{alignat*}
where $n_S$ denotes the outward unit normal of $S$.  We also denote by $\mathring{L}^2(S)$ the space
of square integrable functions on $S$ with vanishing mean.

For $r\in \bbN$, let $\pol_r(S)$ denote the space of polynomials 
of degree $\le r$ with domain $S$, and we use the convention $\pol_r(S) = \{0\}$ for $r < 0$. Define the piecewise polynomial space
on the Powell--Sabin split as
\[
\pol_r(\PST) = \{q\in L^2(T):\ q|_S \in \pol_r(S), \,\forall S\in \PST\}.
\]
\begin{remark}\label{rem:factoring}
For any $q\in \pol_r(\PST)$ satisfying $q|_{\p T}=0$, there exists $p\in \pol_{r-1}(\PST)$ such that
$q = \mu p$.
\end{remark}
\begin{definition}
Let $r\in\bbN$.  The 
{\em N\'ed\'elec spaces} (of the second-kind) with and without boundary conditions are given by \cite{Nedelec86}
\begin{alignat*}{2}
&V_r^0(\PST) = \pol_r(\PST)\cap H(\rote;T),\qquad &&\mathring{V}_r^0(\PST) = \pol_r(\PST)\cap \mathring{H}(\rote;T),\\
&V_r^1(\PST) = \pol_r(\PST)\cap H({\rm div};T),\qquad &&\mathring{V}_r^1(\PST) = \pol_r(\PST)\cap \mathring{H}({\rm div};T),\\
&V_r^2(\PST) = \pol_r(\PST),\qquad && \mathring{V}_r^2(\PST) = \pol_r(\PST)\cap \mathring{L}^2(T).
\end{alignat*}
\end{definition}

\begin{definition}
The {\em Lagrange space} $L_r^k(\PST)$ (resp., $\mathring{L}_r^k(\PST)$) is the subspace
of $V_r^k(\PST)$ (resp., $\mathring{V}_r^k(\PST)$) consisting of continuous piecewise polynomials, i.e., 
\begin{alignat*}{2}
&L_r^0(\PST)= \pol_r(\PST)\cap C(T),\qquad &&\mathring{L}_r^0(\PST)= L_r^0(\PST)\cap \mathring{H}(\rote;T),\\
&L_r^1(\PST)= [L_r^0(\PST)]^2, \qquad &&\mathring{L}_r^1(\PST)=  [\mathring{L}_r^0(\PST)]^2,\\
&L_r^2(\PST)=L_r^0(\PST),\qquad &&\mathring{L}_r^2(\PST)={\mathring{L}_r^0(\PST)\cap \mathring{V}_r^2(\PST)}.
\end{alignat*}
\end{definition}

\begin{remark}
Note the redundancies in notation, $L_r^0(\PST)=V_r^0(\PST)$ and  $\mathring{L}_r^0(\PST)=\mathring{V}_r^0(\PST)$.
\end{remark}

\begin{definition}
We define the {\em smooth spaces}
with and without boundary conditions as
\begin{alignat*}{2}
&S_r^0(\PST)= \{ v \in L_r^0(\PST): \rot v \in [C(T)]^2\}, \quad &&\mathring{S}_r^0(\PST)= \{ v \in S_r^0(\PST): v=0 \text{ and } \rot v=0 \text{ on } \partial T   \},\\
&S_r^1(\PST)= \{ v \in L_r^1(\PST): \dive v \in C(T),\}, \quad &&\mathring{S}_r^1(\PST)= \{ v \in S_r^1(\PST):   v=0 \text{ and } \dive v=0 \text{ on } \partial T   \}, \\
&S_r^2(\PST)=L_r^2(\PST),\quad &&\mathring{S}_r^2(\PST)=\mathring{L}_r^2(\PST).
\end{alignat*}
\end{definition}

\section{Exact sequences on a macro triangle}\label{sec:SequencesMacro}
The goal of this section is to derive
exact sequences consisting of the piecewise polynomial
spaces defined in the previous section.
As a first step, we state a well-known result,
that the N\'ed\'elec spaces form exact sequences \cite{AFW06,AFW10}.
\begin{proposition}\label{prop:NedelecExact}
The following sequences are exact, i.e., 
the range of each map is the kernel of the succeeding map
\begin{alignat*}{5}
&\bbR\
{\xrightarrow{\hspace*{0.5cm}}}\
{V}_{r}^0(\PST)\
&&\stackrel{\rote}{\xrightarrow{\hspace*{0.5cm}}}\
{V}_{r-1}^1(\PST)\
&&\stackrel{\dive}{\xrightarrow{\hspace*{0.5cm}}}\
{V}_{r-2}^2(\PST)\
&&\xrightarrow{\hspace*{0.5cm}}\
 0,\\
 &0\
{\xrightarrow{\hspace*{0.5cm}}}\
\mathring{V}_{r}^0(\PST)\
&&\stackrel{\rote}{\xrightarrow{\hspace*{0.5cm}}}\
\mathring{V}_{r-1}^1(\PST)\
&&\stackrel{\dive}{\xrightarrow{\hspace*{0.5cm}}}\
\mathring{V}_{r-2}^2(\PST)\
&&\xrightarrow{\hspace*{0.5cm}}\
 0.
 \end{alignat*}
\end{proposition}

The goal now is to extend Proposition \ref{prop:NedelecExact}
to incorporate smooth spaces.  An integral component
of this extension is a characterization of the range of the divergence operator
acting on the (vector-valued) Lagrange space.
For example, it is known \cite[Proposition 2.1]{ScottVogelius85}
that if $v \in \mathring{L}_r^1(\PST)$ then $\dive v$ is continuous at 
the vertices $z_4, z_5, z_6$.  In particular, this is because each of these vertices is a {\em singular vertex},
i.e., the edges meeting at the vertex fall on exactly two straight lines.
Hence, in order to extend Proposition \ref{prop:NedelecExact} and
to characterize the range of $\dive \mathring{L}_r^1(\PST)$,
we will consider the spaces
\begin{align*}
\VV_r^2(\PST)&= \{ q\in V_r^2(\PST):  q \text{ is continuous at } z_4, z_5, z_6\},\\
\mathring{\VV}_r^2(\PST)&=  \VV_r^2(\PST)\cap \mathring{L}^2(T).
\end{align*}
We then have that $\dive \mathring{L}_r^1(\PST) \subset \mathring{\VV}_{r-1}^2(\PST)$. 
In this  section we show that 
$\dive:\mathring{L}_r^1(\PST)\to \mathring{\VV}_{r-1}^2(\PST)$ is surjective, i.e.,
$\dive \mathring{L}_r^1(\PST) = \mathring{\VV}_{r-1}^2(\PST)$.

The proof of this result 
is based on several preliminary lemmas.  As a first step, we state
the canonical degrees of freedom for the lowest order 
N\'ed\'elec $H({\rm div})$--conforming finite element space on the unrefined triangulation
\cite{Nedelec86}.
\begin{lemma}\label{lem:NedelecHdiv}
Any $w\in [\pol_1(T)]^2$ is uniquely determined by 
the values
\[
\int_{e_i} (w\cdot n_i)\kappa \quad \forall \kappa\in \pol_1(e_i).
\]
\end{lemma}


\begin{lemma}\label{lemma2}
Let  $q \in \VV_{r}^2(\PST)$  and $r \ge 1$, then there exists  $w \in L_r^1(\PST)$  and $g \in V_{r-1}^2(\PST)$ such that $\mu^s q=\dive (\mu^{s+1} w)+ \mu^{s+1} g$ for any $s \ge 0$. 
\end{lemma}
\begin{proof}
Let  $b_i \in \pol_1(e_i)$ be the linear function such that $q|_{e_i}-b_i$ vanishes at the end points of $e_i$. 
Because $q-b_i$ vanishes at the endpoints and $q$ is continuous at $z_{3+i}$, 
there exists $a_i\in L^0_r(\PST)$ such that $a_i|_{e_i}=(q-b_i)|_{e_i}$
and $\text{ supp } a_i \in T(z_{3+i})$. Note that $a_i|_{e_j} = 0$ for $i\neq j$.

Next, using \eqref{mun} and the N\'ed\'elec degrees of freedom stated
in Lemma \ref{lem:NedelecHdiv},  we construct a unique function $w_1 \in [\pol_1(T)]^2$ such that 
\begin{equation*}
(s+1) w_1 \cdot \nabla \mu_i = b_i \quad \text{on } e_i,\quad i=1,2,3.
\end{equation*}
We set  $\ell_i= \frac{\nab \mu_i}{ |\nab \mu_i|^2}$, 
\begin{equation*}
w_2 = \frac{1}{s+1}(a_1 \ell_1+ a_2 \ell_2+ a_3 \ell_3),\quad \text{and}\quad w=w_1+w_2.
\end{equation*}
We then see that, on $e_i$,
\begin{equation*}
 (s+1)w \cdot \nab \mu_i=   (s+1) w_1\cdot \nab \mu_i + (s+1)  w_2 \cdot \nab \mu_i=  b_i +a_i=  q.
\end{equation*}
Therefore the function $(s+1) w \cdot \nabla \mu -q$ vanishes on $\partial T$, which implies
 that
 $\mu v=(s+1) w \cdot \nabla \mu -q$ for some $v\in V_{r-1}^2(\PST)$; see Remark \ref{rem:factoring}.

Finally we compute
\begin{equation*}
\mu^s q= \mu^s q+\dive (\mu^{s+1} w )- \mu^{s+1} \dive(w) - \mu^{s} (s+1) w \cdot \nabla \mu=  \dive (\mu^{s+1}w)-\mu^{s+1} (\dive(w) + v ).
\end{equation*}
The proof is complete upon setting
$g=-(\dive w+ v)$. 
\end{proof}

\begin{lemma}\label{lemma1}
For any $\theta \in V_{r}^2(\PST)$ with $r \geq 0$, there exists  $\psi \in L_{1}^1(\PST)$ and $\gamma \in \VV_{r}^2(\PST)$
such that
\begin{equation}\label{aux12}
 \mu^s \theta =\dive(\mu^{s} \psi)+ \mu^s \gamma \quad \text{ for any } s \ge 0.
\end{equation}
\end{lemma}
\begin{proof} 
Given $\theta \in V_{r}^2(\PST)$, we define $a_i\in L_1^0(\PST)$ uniquely by 
the conditions
\[
a_i(z_j) = 0,\ \ j=0,1,2,3,\qquad  a_i(z_{3+j})=0,\ \ j\neq i,\qquad \jmp{\nab a_i\cdot t_i}(z_{3+i}) = \jmp{\theta}(z_{3+i}).
\]
We clearly have $\text{supp } a_i \in T(z_{3+i})$. Setting $\psi = a_1t_1+ a_2 t_2+ a_3 t_3$ we have 
\[
\Div \psi|_{e_i} = \nab a_i \cdot t_i,
\]
and therefore, by the construction of $a_i$, $\gamma:=\theta-\Div \psi \in \VV_r^2(\PST)$.
Furthermore, we have $\psi\cdot \nab \mu|_{T(z_{3+i})} = a_i t_i \cdot \nab \mu|_{T(z_{3+i})} = 0$ for $i=1,2,3$ by \eqref{mut}, and so $\psi \cdot \nab \mu =0$ in $T$.  It then follows that
\begin{align*}
\mu^s \theta - \Div (\mu^s \psi) = \mu^s(\theta- \Div \psi) - s \mu^{s-1} \nab \mu \cdot \psi = \mu^s \gamma.
\end{align*}
\end{proof}

We  combine the previous two lemmas to obtain the following. 
\begin{lemma}\label{lemma301}
Let  $q \in \VV_{r}^2(\PST)$  and $r \ge 1$.  Then there exists  $v \in L_r^1(\PST)$  and $Q \in \VV_{r-1}^2(\PST)$ such that $\mu^s q=\dive (\mu^{s+1} v)+ \mu^{s+1} Q$ for any $s \ge 0$. 
\end{lemma}

Our last lemma handles the lowest order case which follows
from \cite[Lemma 3.11]{FuGuzmanNeilan19}.
\begin{lemma}\label{lemma3}
Let  $q \in \VV_{0}^2(\PST)$ with $\int_T \mu^s q =0$.  Then there exists  $w \in L_0^1(\PST)$   such that $\mu^s q=\dive (\mu^{s+1} w)$ for any $s \ge 0$. 
\end{lemma}

We can now state and prove the main result. 
\begin{theorem}\label{mainthm}
For each $p \in \mathring{\VV}_r^2(\PST)$, with $r \geq 0$, there exists a $v \in \mathring{L}_{r+1}^{1}(\PST)$ such that $\dive v=p$.
\end{theorem}
\begin{proof}
Let $p_r=p$ and suppose we have  found $w_{r-j}\in L_{r-j}^1(\PST)$ for $0 \le j \le \ell-1$  and  $p_{r-j} \in \VV_{r-j}^2(\PST)$ for $0 \le j \le  \ell$ such that 
\begin{equation}\label{aux1001}
\dive (\mu^{j+1} w_{r-j})= \mu^j p_{r-j} - \mu^{j+1} p_{r-(j+1)} \qquad \text{ for all } 0 \le j  \le   \ell-1.
\end{equation}

We can then apply Lemma \ref{lemma301} to find   $w_{r-\ell}\in L_{r-\ell}^1(\PST)$ and $p_{r-(\ell+1)} \in \VV_{r- (\ell+1)}^2(\PST)$ such that
\begin{equation}\label{aux1001}
\dive (\mu^{\ell+1} w_{r-\ell})= \mu^\ell p_{r-\ell} - \mu^{\ell+1} p_{r-(\ell+1)}.
\end{equation}
Hence, by induction we  can find $w_{r-j}\in L_{r-j}^1(\PST)$ for $0 \le j \le r-1$  and  $p_{r-j} \in \VV_{r-j}^2(\PST)$ for $0 \le j \le  r$ such that \eqref{aux1001} holds.  Therefore, 
\begin{equation*}
\Div(\mu w_r+ \mu^2 w_{r-1} + \cdots + \mu^{r}w_1)= p- \mu^r p_0.
\end{equation*}
We have that $\int_T \mu^r p_0=0$ and hence by Lemma \ref{lemma3}
we can find $w_0 \in L_0^1(\PST)$ such that $\dive (\mu^{r+1} w_0)= \mu^r p_0$. 
The result follows after setting $v=\mu w_{r}+ \mu^2 w_{r-1}+ \cdots \mu^{r} w_1+ \mu^{r+1} w_0$.
\end{proof}

We have several corollaries that follow from Theorem \ref{mainthm}.
First we show that the analogous result without boundary conditions is satisfied.
\begin{corollary}\label{cor1}
For each $p \in V_r^2(\PST)$ there exists a $v \in L_{r+1}^{1}(\PST)$ such that $\dive v=p$.
\end{corollary}
\begin{proof}
Let $p \in V_r^2(\PST)$. By Lemma \ref{lemma1} there exists $w \in  L_1^1(\PST)$ and  $g \in \VV_r^2(\PST)$
with
\begin{equation*}
p=\dive w+ g.
\end{equation*}
We let $\psi= (\frac1{|T|}\int_T g) \frac{1}{2}x \in L_1^1(\PST)$ and hence $\int_T \dive \psi =\int_T g$. We then have 
\begin{equation*}
p=\dive (w+\psi)+ (g-\dive \psi). 
\end{equation*}
By Theorem \ref{mainthm} there exists a $\theta \in \mathring{L}_{r+1}^1(\PST)$ such that  $\dive \theta= g-\dive \psi$. Therefore, we have 
\begin{equation*}
p=\dive (w+\psi+\theta).
\end{equation*}
The proof is complete after we set $v=w+\psi+\theta$.
\end{proof}
\begin{corollary}\label{cor2}
For each $p \in \mathring{L}_{r}^2(\PST)$ (resp., $p\in L_r^2(\PST)$) there exists 
a $v \in \mathring{S}_{r+1}^{1}(\PST)$ (resp., $v\in S_{r+1}^1(\PST)$) such that $\dive v=p$.
Likewise for each $v \in \mathring{L}_r^1(\PST)$ (resp., $v\in L_r^1(\PST)$)
 there exists a $z \in \mathring{S}_{r+1}^{0}(\PST)$ (resp., $z\in S^0_{r+1}(\PST)$) such that $\rote\, z=v$.
\end{corollary}
\begin{proof}
Let $p \in \mathring{L}_r^2(\PST) \subset \mathring{\VV}_r^2(\PST)$ and we can apply Theorem \ref{mainthm} to find $v \in \mathring{L}_{r+1}^1(\PST)$ such that 
$\dive v= p$. However, clearly  $v \in \mathring{S}_{r+1}^{1}(\PST)$. 

Next, let $v\in \mathring{L}_r^1(\PST)\subset V_r^1(\PST)$ be divergence--free.
Proposition \ref{prop:NedelecExact} shows that there exists $z\in \mathring{V}_r^0(\PST)$
such that $\rot  z = v$.  Since $v$ is continuous and vanishes 
on the boundary, we have $\rote \, z\in [C(T)]^2$ and $z|_{\p T} =0,$ $\rote\, z|_{\p T}=0$.
Thus $z\in \mathring{S}_r^0(\PST)$ by definition.

This proof applies {\em mutatis mutandis} to the  statements without boundary conditions. 
\end{proof}

\begin{remark}
To summarize, Proposition \ref{prop:NedelecExact}, Theorem \ref{mainthm},
and Corollaries \ref{cor1} and \ref{cor2} show that the following two sets of sequences are exact:
\begin{align*}
\bbR \lrarrow &L_r^0(\PST) \overset{\rot}{\lrarrow} V_{r-1}^1(\PST) \overset{\dive}{\lrarrow} V_{r-2}^2(\PST) \lrarrow 0, \\
\bbR \lrarrow &S_r^0(\PST) \overset{\rot}{\lrarrow} L_{r-1}^1(\PST) \overset{\dive}{\lrarrow} V_{r-2}^2(\PST) \lrarrow 0, \\
\bbR \lrarrow &S_r^0(\PST) \overset{\rot}{\lrarrow} S_{r-1}^1(\PST) \overset{\dive}{\lrarrow} L_{r-2}^2(\PST) \lrarrow 0,
\end{align*}
%
and
%
\begin{align*}
0 \lrarrow &\mathring{L}_r^0(\PST) \overset{\rot}{\lrarrow} \mathring{V}_{r-1}^1 (\PST) \overset{\dive}{\lrarrow} \mathring{V}_{r-2}^2(\PST) \lrarrow 0,\\
0 \lrarrow &\mathring{S}_r^0(\PST) \overset{\rot}{\lrarrow} \mathring{L}_{r-1}^1 (\PST) \overset{\dive}{\lrarrow} \mathring{\VV}_{r-2}^2(\PST) \lrarrow 0,\\
0 \lrarrow &\mathring{S}_r^0(\PST) \overset{\rot}{\lrarrow} \mathring{S}_{r-1}^1 (\PST) \overset{\dive}{\lrarrow} \mathring{L}_{r-2}^2(\PST) \lrarrow 0. 
\end{align*}
\end{remark}

\subsection{Dimension Counting}
We can easily count the dimensions of the smooth spaces $S_r^{k}(\PST)$
via the rank--nullity theorem and the exactness of sequences ($k=0,1$):
\begin{alignat*}{1}
\dim S_r^k(\PST)= & \dim \range S_r^k (\PST)+ \dim \ker S_r^k(\PST) \\
= & \dim \ker L_{r-1}^{k+1}(\PST)+   \dim \ker L_r^k(\PST) \\
=& \dim L_{r-1}^{k+1}(\PST)-  \dim \range L_{r-1}^{k+1}(\PST) +  \dim   L_r^k(\PST)- \range L_r^k(\PST) \\
= & \dim L_{r-1}^{k+1}(\PST)+ \dim L_r^k(\PST) -   \dim \ker  V_{r-2}^{k+2}(\PST) -  \dim \ker V_{r-1}^{k+1}(\PST)  \\
= & \dim L_{r-1}^{k+1}(\PST)+ \dim L_r^k(\PST) -  \dim V_{r-1}^{k+1} (\PST).
\end{alignat*}
Now we easily find
\begin{align*}
\dim L_r^k(\PST) 
%
%
 = \binom{2}{k}\big[3r^2+3r+1\big],\quad
 \dim V_r^k(\PST) 
= \left\{
\begin{array}{ll}
3r^2+3r+1 & k=0,\\
6r^2+12r+6 & k=1,\\
3r^2+9r+6 & k=2.
\end{array}
\right.
\end{align*}
Thus, we have
\[
\dim S_r^k(\PST)
 = \left\{
\begin{array}{ll}
3r^2-3r+3 & k=0,\\
6r^2+3 & k=1,\\
3r^2+3r+1 &k=2.
\end{array}
\right.
\]

%
%
Similar calculations also show that
\begin{align*}
\dim \mathring{S}_r^k(\PST) &= 
\left\{
\begin{array}{ll}
3(r-2)(r-3) & k=0,\\
{6(r-1)(r-2)} & k=1,\\
{3r(r-1)}  &k=2.
\end{array}
\right.
\end{align*}

\section{Commuting Projections on a Macro Triangle}\label{sec:commuting_projections}
In this section we define commuting projections. In order to do so, we give the 
degrees of freedom for $C^1$ polynomials
on a line segment. Let $a < m < b$, and define the space
\begin{equation*}
W_r(\{a, m, b\})= \{ v \in C^1([a, b]): v|_{[a, m]} \in \pol_r([a, m]) \text{ on } v|_{[m, b]} \in \pol_r([m, b]) ).   
\end{equation*}
The classical degrees of freedom for $W_r(\{a, m, b\})$ is given in the next result. 
\begin{lemma}\label{lem:edge1}
Let $r \ge 1$.  A function $z \in W_r(\{a,m,b\})$ is uniquely determined by the following degrees of freedom. 
\begin{alignat*}{1}
z(a), z(b) & \\
z'(a), z' (b) & \quad \text{ if } r \ge 2, \\
z(m), z'(m)  & \quad \text{ if } r \ge 3, \\
\int_{a}^m z(x) q(x)  & \quad \text{for all } q \in \pol_{r-4}([a,m]), \\
\int_{m}^b z(x) q(x)  & \quad \text{for all } q \in \pol_{r-4}([m,b]) .  
\end{alignat*}
\end{lemma}

Other degrees of freedom are given in the next lemma.
Its proof is found in the appendix.
\begin{lemma}\label{lem:edge2}
Let $r \ge 1$.  A function $z \in W_r(\{a, m, b\})$ is uniquely determined by the following degrees of freedom. 
\begin{subequations}
\begin{alignat}{1}
z(a), z(b) & \label{1d:rgeq1}\\
\int_{a}^m z(x) q(x) & \quad \text{for all } q \in \pol_{r-2}([a,m]), \label{1d:rgeq2a}\\
\int_{m}^b z(x) q(x) & \quad \text{for all } q \in \pol_{r-2}([m,b]) .\label{1d:rgeq2b}  
\end{alignat}
\end{subequations}
\end{lemma}

\begin{lemma}\label{lem:FactorC1}
Suppose that $q\in S_r^0(\PST)$ with $q|_{\p T}=0$.
Then $q = \mu p$ for some $p\in L_{r-1}^0(\PST)$,
and $p|_{T(z_{3+i})}\in C^1(T(z_{3+i}))$, $i=1,2,3$.
\end{lemma}
\begin{proof}
The statement $q = \mu p$ is a direct consequence of Remark \ref{rem:factoring}.
Because $q$ and $\mu$ are continuous, 
it follows that $p$ is continuous, i.e., $p\in L_{r-1}^0(\PST)$.
We also have $\nab q = \mu \nab p+ p\nab \mu$, and therefore
\[
\mu \nab p|_{T(z_{3+i})} = \big(\nab q - p \nab\mu \big)|_{T(z_{3+i})}.
\]
Since $\nab \mu$ is constant on $T(z_{3+i})$, we find
that $\mu \nab p|_{T(z_{3+i})}$ is continuous. Because
$\mu$ is positive in the interior of $T(z_{3+i})$, we conclude
that $\nab p$ is continuous on $T(z_{3+i})$.
\end{proof}

We are now ready to give degrees of freedom (DOFs) for functions in $S_r^0(\PST)$. 
\begin{lemma}\label{lem:Sr0}
A function $q \in S_r^0(\PST)$, with $r\ge 2$, is uniquely determined by 
\begin{subequations}\label{S0r}
\begin{alignat}{3}
&q(z_i), \nab q(z_i)  \qquad  && 1 \le i \le 3, \quad &&\text{($9$ DOFs)}\label{S0r:a}\\
&q(z_{3+i}), \partial_t q(z_{3+i})  \qquad  &&1 \le i \le 3, \text{ if } r \ge 3,\quad &&\text{($6$ DOFs)} \label{S0r:c}\\
&\int_{e} \partial_n q \, p  \qquad  &&\forall p \in \pol_{r-3}(e),\ e \in \mathcal{E}^b(\PST),\quad &&\text{($6(r-2)$ DOFs)}  \label{S0r:d}\\
&\int_{e } q  p  \qquad   &&\forall p \in \pol_{r-4}(e),\ e \in \mathcal{E}^b(\PST),\quad &&\text{($6(r-3)$ DOFs)}  \label{S0r:e}\\
&\int_{T} \rot q  \cdot \rot p  \qquad &&\forall p \in \mathring{S}_{r}^0(\PST),\quad &&{\text{($3(r-2)(r-3)$ DOFs)}} \label{S0r:f}
\end{alignat}
\end{subequations}
\end{lemma}
\begin{proof}
{The number of DOFs given is $3r^2-3r+3 = \dim S_r^0(\PST)$.}
We will show that the only function $q$ for which \eqref{S0r:a}--\eqref{S0r:f} are equal to zero must be zero on $T$.

Combining \eqref{S0r:a},  \eqref{S0r:c} and \eqref{S0r:e}, $q$ satisfies all conditions of Lemma \ref{lem:edge1} on each edge of $T$, so $q \equiv 0$ on $\partial T$. By Lemma \ref{lem:FactorC1},  there exists a $p$ which is a piecewise polynomial of degree $r-1$,
and is $C^1$ on edges, such that $q = \mu p$, and $\nabla q|_{\p T} = (p \nabla \mu + \mu \nabla p)|_{\p T} = p \nab \mu|_{\p T}$. 
Then \eqref{S0r:a} yields $p(z_i) = 0$ for $1 \leq i \leq 3$.  
Also \eqref{S0r:d} yields $\int_e p w\partial_n \mu = 0$ for all $w \in \pol_{r-3}(e)$ and for all $e \in \mathcal{E}^b(\PST)$. Since $\partial_n \mu$ is constant on each edge $e \in \mathcal{E}^b(\PST)$, we have $\int_e p w = 0$ for all $w \in \pol_{r-3}(e)$, and using Lemma \ref{lem:edge2}, it follows that $p \equiv 0$ on $\partial T$.  Thus $q\in \mathring{S}_r^0(\PST)$
and condition \eqref{S0r:f} yields $\rot q = 0$ on $T$. Then $q$ is constant and so must be equal to zero on $T$.
\end{proof}

\begin{lemma}\label{lem:Lr1DOFs}
A function $v \in L_r^1(\PST)$ is uniquely determined by 
\begin{subequations}\label{L1r}
\begin{alignat}{3}
v(z_i), & \qquad  && 1 \le i \le 3, \label{L1r:a}\quad &&\text{($6$ DOFs)},\\
\int_{e_i} (v \cdot n_i)  & \qquad  && \text{ if } r=1, \label{L1r:b}\\
\jmp{\dive v}(z_{3+i}) & \qquad  && 1 \le i \le 3, \label{L1r:c}\quad && \text{($3$ DOFs)},\\
v(z_{3+i}) \cdot n_i   & \qquad  && 1 \le i \le 3,  \text{ if } r \ge 2, \label{L1r:d}\quad &&\text{($3$ DOFs)},\\
\int_{e} v \cdot w & \qquad  &&  \forall w \in [\pol_{r-2}(e)]^2, \text{ for all } e \in \mathcal{E}^b(\PST),  \label{L1r:e}\quad &&\text{($12(r-1)$ DOFs)},\\%
\int_{T} v \cdot \rot w & \qquad && \forall w \in \mathring{S}_{r+1}^0(\PST) \label{L1r:g},\quad &&{\text{($3(r-1)(r-2)$ DOFs)}},\\
\int_{T} \dive v \, w & \qquad && \forall w \in \mathring{\VV}_{r-1}^2(\PST) \label{L1r:h},\quad &&{\text{($3r(r+1)-4$ DOFs)}}.
\end{alignat}
\end{subequations}
\end{lemma}
\begin{proof}
The number of degrees of freedom given is $6r^2+6r+2$ which equals the dimension
of $L_r^1(\PST)$.  We show that if $v\in L_r^1(\PST)$ vanishes on \eqref{L1r}, then $v$ is identically zero.

Recall that $T(z_{3+i}) = T_{2i+1}\cup T_{2i+2}$ is the union of
two triangles that have $z_{3+i}$ as a vertex, and $n_i$
and $t_i$ are, respectively, the outward normal and unit tangent
vectors of the edge $e_i = \p T\cap \p T(z_{3+i})$.
Let $s_i$ be a unit vector that is tangent to the interior 
edge $[z_0,z_{3+i}]$, which is necessarily linearly independent of $t_i$.
Thus we may write
\[
v|_{T(z_{3+i})} = a_i t_i + b_i s_i
\]
for some $a_i,b_i\in L_r^0(\PST)|_{T(z_{3+i})}$.  We then see that 
\begin{align*}
\Div v|_{T(z_{i+3})} = \p_{t_i} a_i + \p_{s_i} b_i.
\end{align*}
Because $b_i$ is continuous on $T(z_{3+i})$ we have that  $\jmp{\p_{s_i} b_i}(z_{i+3})=0$ and hence $0=\jmp{\Div v} (z_{3+i}) = [\p_{t_i} a_i](z_{3+i})$. 
Therefore $a_i|_{e_i}$ is $C^1$ on $e_i$.  To continue, we  split the proof into two step.

\noindent{\em Case $r=1$:}\\
By the first set of DOFs \eqref{L1r:a}, there holds $a_i(z_j) = b_i(z_j)=0$
for $j=i+1,i+2$.  Because $a_i|_{e_i}$ is piecewise linear and $C^1$,
we conclude that $a_i=0$ on $e_i$.
Next,  using \eqref{L1r:b} yields
\[
\int_{e_i} b_i (s_i\cdot n_i) = 0.
\]
Because $s_i\cdot n_i\neq 0$, we conclude that $\int_{e_i} b_i = 0$.
Since $b_i$ vanishes at the endpoints of $e_i$, and since $b_i$ is piecewise
linear on $e_i$, we conclude that $b_i=0$ on $e_i$, and therefore $v|_{\p T}=0$,
i.e., $v\in \mathring{L}_1^1(\PST)$.  Corollary \ref{cor1} and \eqref{L1r:h} 
then shows that $\Div v\equiv 0$.  Finally, Corollary \ref{cor2} and \eqref{L1r:g} yields $v\equiv 0$.\smallskip

\noindent{\em Case $r\ge 2$:}\\
Again, there holds $a_i(z_j) = b_i(z_j)=0$
by the first set of DOFs \eqref{L1r:a}.  
Combining Lemma \ref{lem:edge2} 
with the DOFs \eqref{L1r:e}, noting that $a_i$ is $C^1$ on $e_i$,
then yields $a_i=0$ on $e_i$.  Likewise
the DOFs \eqref{L1r:a}, \eqref{L1r:e}, and \eqref{L1r:d}
show that $b_i=0$ on $e_i$.  We conclude that $v|_{\p T}=0$,
and therefore, using \eqref{L1r:g}--\eqref{L1r:h} and the same arguments as the $r=1$ case,
we get $v\equiv 0$.
\end{proof}

\begin{lemma}\label{lem:V2r}
A function $q \in V_r^2(\PST)$ is uniquely determined by 
\begin{subequations}\label{V2r}
\begin{alignat}{3}
\jmp{ q  }(z_{3+i})  & \qquad  && 1 \le i \le 3,&& \qquad \text{($3$ DOFs)},  \label{V2r:a} \\
\int_T q  & \qquad && \quad && \qquad \text{(1 DOF)}, \label{V2r:b} \\
\int_T q p   & \qquad && \forall p \in \mathring{\VV}_r^2(\PST), && \qquad {\text{($3(r+1)(r+2)-4$ DOFs)}}, \label{V2r:c}
\end{alignat}
\end{subequations}
\end{lemma}
\begin{proof}
If $q \in V_r^2(\PST)$ is such that \eqref{V2r:a} are zero then $q$ is continuous at $z_{3+i}$ for $1 \leq i \leq 3$. Then \eqref{V2r:b} yields that $q \in \mathring{\VV}_r^2(\PST)$, and it follows from \eqref{V2r:c} that $q \equiv 0$ on $T$.
\end{proof}

\begin{lemma}\label{lem:S1r}
A function $v \in S_r^1(\PST)$ is uniquely determined by the following degrees of freedom.
\begin{subequations}
\label{S1r}
\begin{alignat}{3}
&v(z_i),\ \Div v(z_i)\qquad &&1\le i\le 3,\qquad &&\text{($9$ DOFs)},\label{S1r:a} \\
&\int_{e_i} v\cdot n_i\qquad &&1 \le i \le 3, \ \text{if $r=1$},\label{S1r:b}\\
&v(z_{3+i})\cdot n,\ \Div v(z_{3+i})\qquad &&1\le i\le 3,\ \text{if $r\ge 2$}\qquad &&\text{($6$ DOFs)},\label{S1r:c} \\
&\int_e v\cdot w\qquad &&\forall w\in [\pol_{r-2}(e)]^2,\ e\in \mathcal{E}^b(\PST),\quad &&\text{($12(r-1)$ DOFs)},\label{S1r:d} \\
&\int_e (\Div v)q\qquad &&\forall q\in \pol_{r-3}(e),\ e\in \mathcal{E}^b(\PST),\quad &&\text{($6(r-2)$ DOFs)},\label{S1r:e}\\
&\int_T  v \cdot \rot q,  \quad &&\forall q \in \mathring{S}_{r+1}^0(\PST)\qquad &&{\text{($3(r-1)(r-2)$ DOFs)}},\label{S1r:f}\\
&\int_T (\dive v)  q,  \quad &&\forall q \in \mathring{L}_{r-1}^2(\PST)\qquad &&{\text{($3(r-1)(r-2)$ DOFs)}.\label{S1r:g}}
\end{alignat}
\end{subequations}
\end{lemma}
\begin{proof}
If $v$ vanishes at the DOFs, then $v\in S^1_r(\PST)\subset L^1_r(\PST)$
vanishes on \eqref{L1r:a}--\eqref{L1r:e}.  The proof
of Lemma \ref{lem:Lr1DOFs} then shows
that $v|_{\p T}=0$, and therefore $\int_T \Div v = 0$.
Using \eqref{S1r:a},\eqref{S1r:c}, and \eqref{S1r:e},
we also find that $\Div v|_{\p T}=0$, i.e., $\Div v\in \mathring{L}_{r-1}^2(\PST)$.
The DOFs \eqref{S1r:g} yield $\Div v=0$ in $T$, and therefore $v = \rot q$ for some $q\in \mathring{S}_{r+1}^0(\PST)$
by Corollary \ref{cor2}.  Finally \eqref{S1r:f} gives $v\equiv 0$.
{Noting that the number of DOFs is $6r^2+3$, the dimension
of $S_r^1(\PST)$, we conclude that \eqref{S1r} form
a unisolvent set over $S_r^1(\PST)$.}
\end{proof}

\begin{lemma}\label{lem:Lr2}
Let $q \in L_r^2(\PST)$ with $r \ge 1$. Then $q$ is uniquely determined by the following degrees of freedom.
\begin{subequations}\label{Lr2}
\begin{alignat}{3}
&q(z_i) \quad &&1 \leq i \leq 3, \qquad &&\qquad \text{($3$ DOFs),}\label{Lr2:a}\\
&q(z_{3+i}) \quad && 1 \leq i \leq 3, &&\qquad \text{($3$ DOFs),}\label{Lr2:b}\\
&\int_e qp  \quad &&\forall p \in \pol_{r-2}(e), \quad  e \in \mathcal{E}^b(\PST), && \qquad \text{($6(r-1)$ DOFs),}\label{Lr2:c}\\
&\int_T q \quad && &&\qquad \text{($1$ DOF),}\label{Lr2:c2}\\
&\int_T qp   \quad && \forall p \in \mathring{L}_r^2(\PST), &&\qquad{ \text{($3r(r-1)$ DOFs)}}. \label{Lr2:d}
\end{alignat}
\end{subequations}
\end{lemma}
\begin{proof}
Let $q \in L_r^2(\PST)$ such that all DOFs \eqref{Lr2} are equal to zero. The conditions \eqref{Lr2:a}--\eqref{Lr2:c} yield that $q \equiv 0$ on $\partial T$. Therefore, using \eqref{Lr2:c2}, $q \in \mathring{L}_r^2(\PST)$, and by \eqref{Lr2:d}, $q \equiv 0$ on $T$.
\end{proof}

The next two theorems 
show that  projections induced by the degrees
of freedom given in Lemmas \ref{lem:Sr0}--\ref{lem:Lr2}
commute.  
\begin{theorem}\label{thm:Commute1}
Let $\Pi_0^r : C^\infty(T) \rightarrow S_r^0(\PST)$ be the projection induced by the DOFs \eqref{S0r}, that is,
\begin{align*}
\phi(\Pi_0^r p) &= \phi(p), \hspace{15pt} \forall \phi \in \text{DOFs in \eqref{S0r}}.
\end{align*}
Likewise, let $\Pi_1^{r-1} : {[C^\infty(T)]}^2 \rightarrow L_{r-1}^1(\PST)$ be the projection induced by the DOFs \eqref{L1r}, and let $\Pi_2^{r-2} : C^\infty(T) \rightarrow V_{r-2}^2(\PST)$ be the projection induced by the DOFs \eqref{V2r}. Then for $r \geq 2$, the following diagram commutes
\begin{center}
\begin{tikzcd}
  \mathbb{R} \arrow[r] 
    & C^\infty(T) \arrow[d, "\Pi^{r}_0"] \arrow [r,"\rot"] & {[C^\infty(T)]}^2 \arrow[d,"\Pi^{r-1}_1"] \arrow[r,"\dive"] & C^\infty(T) \arrow[r] \arrow[d,"\Pi^{r-2}_2"]& 0 \\
  \mathbb{R} \arrow[r] & S_{r}^0(\PST) \arrow[r,"\rot"] & L_{r-1}^1(\PST) \arrow[r,"\dive"] & V_{r-2}^2(\PST) \arrow[r] & 0.
 \end{tikzcd}
 \end{center}
In other words, we have for $r \ge 2$ 
\begin{subequations}\label{commut:nobdry}
\begin{alignat}{2}
\dive \Pi_1^{r-1}v = & \Pi_2^{r-2} \dive v, \quad && \forall v \in {[C^\infty(T)]^2}, \label{L1divV2} \\
\rot \Pi_0^r p = &\Pi_1^{r-1}  \rot p, \quad && \forall p \in C^\infty(T).\label{S0rotL1}
\end{alignat}
\end{subequations}
\end{theorem}

\begin{proof}

(i) \textit{Proof of} \eqref{L1divV2}. We take $v \in {[C^\infty(T)]^2}$. Since $\rho := \dive \Pi_1^{r-1} v - \Pi_2^{r-2} \dive v \in V^2_{r-2}(\PST)$, we only need to prove that $\rho$ vanishes at the DOFs \eqref{V2r}.   For the jump condition at points $z_{3+i}$ for $1 \leq i \leq 3$, we have
\begin{alignat*}{1}
\jmp{\rho}(z_{3+i}) &= \jmp{ \dive \Pi_1^{r-1} v - \Pi_2^{r-2} \dive v} (z_{3+i}) = \jmp{\dive \Pi_1^{r-1}v - \dive v}(z_{3+i})  = 0,
\end{alignat*}
where we have used the definitions of $\Pi_2^{r-2}$ and $\Pi_1^{r-1}$ along with the DOFs \eqref{V2r:a} and \eqref{L1r:c}.

For the interior DOFs, we have,
\begin{align*}
\int_T \rho 
&= \int_T \left(\dive \Pi_1^{r-1} v - \dive v\right) 
= \int_{\partial T} \left(\Pi_1^{r-1} v - v \right)\cdot n =0,
\end{align*}
where we have used the definitions of $\Pi_1^{r-1}$ and $\Pi_2^{r-2}$ and DOFs \eqref{V2r:b} and either \eqref{L1r:b} if $r = 2$ or \eqref{L1r:e} if $r \geq 3$. 
Finally, for any $p \in \mathring{\VV}_{r-2}^2(\PST)$,
\begin{align*}
\int_T \rho p  &= \int_T \left(\dive \Pi_1^{r-1} v - \Pi_2^{r-1} \dive v \right) p =0 
\end{align*}
by the definitions of $\Pi_1^{r-1}$ and $\Pi_2^{r-2}$ along with DOFs \eqref{V2r:c} and \eqref{L1r:h}. By Lemma \ref{lem:V2r}, $\rho$ is exactly zero on $T$, and the projections in \eqref{L1divV2} commute.\smallskip

(ii) \textit{Proof of} \eqref{S0rotL1}. Let $p \in C^\infty(T)$ and set $\rho := \rot \Pi_0^r p - \Pi_1^{r-1} \rot p \in L_{r-1}^1(\PST)$.
We will show that $\rho$ vanishes for all DOFs \eqref{L1r}.

First, for each vertex $z_i$ with $1 \leq i \leq 3$, 
\begin{align}\label{vertices0}
\rho(z_i) &= \rot \Pi_0^r p(z_i) - \Pi^{r-1}_1 \rot p(z_i) = \rot p(z_i) - \Pi^{r-1}_1 \rot p(z_i) = 0,
\end{align}
by \eqref{S0r:a} and \eqref{L1r:a}. Furthermore, at nodes $z_{3+i}$, we have by \eqref{L1r:c}
\begin{align*}
\jmp{\dive \rho}(z_{3+i}) &= \jmp{\dive \rot \Pi^r_0 p - \dive \Pi^{r-1}_1 \rot p}(z_{3+i}) \\
&= -\jmp{\dive \Pi^{r-1}_1 \rot p}(z_{3+i}) \\
&= -\jmp{\dive \rot p}(z_{3+i}) =0,
\end{align*}

For the DOFs on each edge $e \in \mathcal{E}^b(\PST)$, 
we will use that $\rot \varphi \cdot n = \partial_t \varphi$ and $\rot \varphi \cdot t = -\partial_n \varphi$.
Then we have, for $r\ge 3$,
\begin{align}\label{midpts0}
\begin{split}
	\rho(z_{3+i}) \cdot n_i &= \left(\rot \Pi_0^r p(z_{3+i})\right) \cdot n_i - \left(\Pi^{r-1}_1 \rot p(z_{3+i})\right) \cdot n_i \\
	&= \partial_t p(z_{3+i}) - \left(\Pi_1^{r-1} \rot p(z_{3+i})\right) \cdot n_i \\
	&= \partial_t p(z_{3+i}) - \rot p(z_{3+i})\cdot n_i =0 
\end{split}
\end{align}
by \eqref{S0r:c} and \eqref{L1r:d}. If $r = 2$ (so that $\rho \in L_1^1(\PST)$),
\begin{align*}
\int_{e_i} \rho \cdot n_i &= \int_{e_i} \left(\rot \Pi_r^0 p - \Pi_{r-1}^1 \rot p\right) \cdot n_i
= \int_{e_i} \partial_{t_i} \left(\Pi_r^0 p - p\right)=0
\end{align*}
by \eqref{L1r:b} and \eqref{S0r:a}, so \eqref{S0rotL1} is proved. 

Now let $r \geq 3$. We have, for all $q \in \pol_{r-3}(e)$ and for all $e \in \mathcal{E}^b(\PST)$, 
\begin{align*}
	\int_{e} (\rho \cdot n) q 
	&= \int_e \left(\rot \left(\Pi_0^r p - p\right) \cdot n \right) q \\
	&= \int_e \partial_t \left(\Pi_0^r p - p\right) q \\
	&= -\int_e \left(\Pi_0^r p - p\right) \partial_t q=0,
\end{align*}
by \eqref{L1r:e},  \eqref{S0r:c} and \eqref{S0r:e}. Likewise, for $q \in \pol_{r-3}(e)$,
\begin{align*}
	\int_e (\rho \cdot t) q &= \int_e \left(\left(\rot \Pi_0^r p - \Pi_1^{r-1} \rot p\right) \cdot t \right)q \\
	&= \int_e \left(\rot( \Pi_0^r p - p) \cdot t\right)  q\\
	&= \int_e -\partial_n \left(\Pi_0^r p - p\right) q=0
\end{align*}
by \eqref{L1r:e} and \eqref{S0r:d}.
For the interior DOFs, for any $w \in \mathring{S}_{r-1}^0(\PST)$, we have
\begin{align*}
\int_T \rho \cdot \rot w &= \int_T \left( \rot \Pi_0^r p - \Pi_1^{r-1} \rot p \right) \cdot \rot w =0 
\end{align*}
by \eqref{S0r:f} and \eqref{L1r:g}. Finally, for any $w \in \mathring{\VV}_{r-2}^2(\PST)$,
\begin{align*}
	\int_T \dive \rho \, w &= \int_T \dive\left(\rot \Pi_0^r p - \Pi_1^{r-1} \rot p \right) w  \\
	&= \int_T -\dive (\rot p) w  =0 
\end{align*}
where we used the DOF \eqref{L1r:h}. Therefore $\rho$ is equal to zero on $T$, and the identity \eqref{S0rotL1} is proved.
\end{proof}


The proof of the following result can be found in the appendix. 
\begin{theorem}\label{thm:Commute2}
Let $\Pi_0^r : C^\infty(T) \rightarrow S_r^0(\PST)$ be the projection induced by the DOFs \eqref{S0r}, that is,
\begin{align*}
\phi(\Pi_0^r p) &= \phi(p), \hspace{15pt} \forall \phi \in \text{DOFs in \eqref{S0r}}.
\end{align*}
Likewise, let $\varpi_1^{r-1} : {[C^\infty(T)]}^2 \rightarrow S_{r-1}^1(\PST)$ be the projection induced by the DOFs \eqref{S1r}, and let $\varpi_2^{r-2} : C^\infty(T) \rightarrow L_{r-2}^2(\PST)$ be the projection induced by the DOFs \eqref{Lr2}. Then for $r \geq 2$, the following diagram commutes
\begin{center}
\begin{tikzcd}
  \mathbb{R} \arrow[r] 
    & C^\infty(T) \arrow[d, "\Pi^{r}_0"] \arrow [r,"\rot"] & {[C^\infty(T)]}^2 \arrow[d,"\varpi^{r-1}_1"] \arrow[r,"\dive"] & C^\infty(T) \arrow[r] \arrow[d,"\varpi^{r-2}_2"]& 0 \\
  \mathbb{R} \arrow[r] & S_{r}^0(\PST) \arrow[r,"\rot"] & S_{r-1}^1(\PST) \arrow[r,"\dive"] & L_{r-2}^2(\PST) \arrow[r] & 0.
 \end{tikzcd}
 \end{center}
In other words, we have for $r \ge 2$ 
\begin{subequations}\label{commut:nobdry}
\begin{alignat}{2}
\rot \Pi_0^r p = &\varpi_1^{r-1}  \rot p, \quad && \forall p \in C^\infty(T), \label{S0rotS1} \\
\dive \varpi_1^{r-1}v = & \varpi_2^{r-2} \dive v, \quad && \forall v \in {[C^\infty(T)]^2}. \label{S1divL2}
\end{alignat}
\end{subequations}
\end{theorem}

\section{Global Spaces}\label{sec:Global}


In this section, we study the global finite element spaces induced by the degrees of freedom in Section \ref{sec:commuting_projections}. We let $\mct$ represent the simplicial triangulation of the {polygonal} domain $\Omega \subset \mathbb{R}^2$, and $\mctps$ represent the Powell-Sabin refinement of $\mct$, as discussed in the introduction. We define the set $\mathcal{M}(\mctps)$ to be the points of intersection of the edges of $\mct$ with the edges that adjoin interior points. We also let $\mathcal{E}^b(\mctps)$ be the collection of all the new edges of $\mctps$ that were obtained by sub-dividing edges of $\mct$. We let $\mathcal{E}(\mct)$ be the edges of $\mct$. By the construction of $\mctps$ every $x \in \mathcal{M}(\mctps)$ belongs to edges that lie on two straight lines. Therefore, these vertices are singular vertices \cite{ScottVogelius85}. 
It is important to note that to make our global spaces to have the correct continuity it is essential to construct the meshes in such a way
\cite{SplineBook,PowellSabin77}.
Furthermore, as previously mentioned, the divergence of continuous, piecewise polynomials
have a weak continuity property at singular vertices, i.e., at the vertices in $\mathcal{M}(\mctps)$.  In detail,
  let   $z \in \mathcal{M}(\mctps)$ and suppose that $z$ is an interior vertex. 
 Then it is a vertex of four triangles $K_1, \ldots, K_4 \in \mctps$.  For a function $q$ we define
 \begin{equation*}
 \theta_z(q) := |q|_{K_1}(z) -q|_{K_2}(z)+ q|_{K_3}(z)-q|_{K_4}(z)|.
 \end{equation*}
Then, if $v$ is a continuous piecewise polynomial 
 with respect to $\mctps$, there holds $\theta_z(\Div v) = 0$ \cite{ScottVogelius85}.

The degrees of freedom stated in  Lemmas \ref{lem:Sr0}--\ref{lem:Lr2}
 induce the following spaces
 \begin{alignat*}{1}
  S_r^0 (\mctps)=&\{ q \in C^1(\Omega): q|_T  \in S_r^0(\PST) \, \forall T \in \mathcal{T}_h \}, \\
  S_{r}^1(\mctps) = &\{ v \in [C(\Omega)]^2 : \dive v \in C(\Omega), v|_T \in S_{r}^1(\PST) \, \forall T \in \mct \}, \\
  L_{r}^1(\mctps) = &\{v \in [C(\Omega)]^2 : v|_T \in L_{r}^1(\PST) \, \forall T \in \mct\}, \\
  L_{r}^2(\mctps) = &\{p \in C(\Omega) : p|_T \in L_{r}^2 (\PST) \, \forall T \in \mct \}, \\
  \VV_{r}^2(\mctps) = &\{p \in L^2(\Omega) : p|_T \in V_{r}^2(\PST) \, \forall T \in \mct, \, \theta_z(p) = 0, \ \forall z \in \mathcal{M}(\mctps) \text{ and } z \text{ an interior node}\}.
 \end{alignat*}
 \begin{remark}
 Let $z \in \mathcal{M}(\mctps)$ be an interior vertex and $T_1, T_2 \in \mct$ share a common edge where $z$ lies. 
 Then  $\theta_z(q) = 0$ if and only if $\jmp{q_1}(z)=\jmp{q_2}(z)$ where $q_i=q|_{T_i}$. Therefore, the local degrees of freedom for $V_r^2(\PST)$ with the jump condition \eqref{V2r:a} do indeed induce the global space $\VV_r^2(\mctps)$ above.
 \end{remark}
 
 We list the degrees of freedom of these spaces.  The global DOF come directly from the local DOF. We list them here to be precise.

It follows from Lemma \ref{lem:Sr0} that a function $q \in S_r^0(\mctps)$, with $r \geq 2$, is uniquely determined by
\begin{alignat*}{2}
q(z),\ \nab q(z) & \qquad  && \text{for every vertex } z \text{ of } \mct,\\
q(z), \partial_t q(z) & \qquad  && \forall z \in \mathcal{M}(\mathcal{T}_h^{\PS}),   \text{ if } r \ge 3, \\
\int_{e} \partial_n q \, p & \qquad  &&  \forall p \in \pol_{r-3}(e), \text{ for all } e \in \mathcal{E}^b( \mathcal{T}_h^{\PS}) \\
\int_{e } q  p & \qquad  && \forall p \in \pol_{r-4}(e), \text{ for all } e \in \mathcal{E}^b( \mathcal{T}_h^{\PS}),  \\
\int_{T} \rot q  \cdot \rot p & \qquad && \forall p \in \mathring{S}_{r}^0(\PST), \text{ for all } T \in \mathcal{T}_h.
\end{alignat*}
\begin{remark}
The degrees of freedom for $r=2$ coincide with the known degrees of freedom of Powell-Sabin \cite{PowellSabin77,SplineBook}. 
Recently, results for polynomial degrees $r=3, 4$ have appeared \cite{Gros16A,Gros16B}. 
\end{remark}

Lemma \ref{lem:Lr1DOFs} shows that a function $v \in L_r^1(\mctps)$ is uniquely determined by the values
\begin{alignat*}{2}
v(z), & \qquad && \text{for every vertex } z \text{ of } \mct,\\
\int_{e} (v \cdot n), & \qquad && \forall e \in \mathcal{E}(\mct), \text{ if } r = 1,\\
\jmp{\dive v}(z), & \qquad && \forall z \in \mathcal{M}(\mctps), \\
v(z) \cdot n, & \qquad &&\forall z \in \mathcal{M}(\mctps), \text{ if } r \ge 2,\\
\int_e v \cdot w, & \qquad && \forall w \in [\pol_{r-2}(e)]^2, \ \forall e \in \mathcal{E}^b (\mctps),\\
\int_T v \cdot \rot w, & \qquad && \forall w \in \mathring{S}_{r+1}^0(\PST), \forall T \in \mct, \\
\int_T \dive v \, w, & \qquad && \forall w \in \mathring{\VV}_{r-1}^2(\PST), \forall T \in \mct.
\end{alignat*}

A function $q\in \VV_r^2(\mctps)$, for $r \ge 0$, is uniquely determined by
\begin{alignat*}{2}
\jmp{q}(z), & \qquad && \forall z \in \mathcal{M}(\mctps), \\
\int_T q = 0, & \qquad && \forall T \in \mct, \\
\int_T qp & \qquad && \forall p \in \mathring{\VV}_r^2(\mctps), \forall T \in \mct.
\end{alignat*}

A function $v \in S_r^1(\mctps)$ is determined by the following degrees of freedom.
\begin{alignat*}{2}
v(z), \ \dive v(z) & \qquad && \text{for every vertex } z \text{ of } \mct, \\
\int_{e} (v \cdot n_i), & \qquad && \forall e \in \mathcal{E}(\mct), \text{if } r = 1, \\
v(z) \cdot n, \ \dive v(z) & \qquad && \forall z \in \mathcal{M}(\mctps), \text{ if } r \ge 2,  \\
\int_e v \cdot w & \qquad && \forall w \in [\pol_{r-2}(e)]^2, \ e \in \mathcal{E}^b(\mctps), \\
\int_e (\dive v) q & \qquad && \forall q \in \pol_{r-3}(e), \ e \in \mathcal{E}^b(\mctps), \\
\int_T v \cdot \rot w & \qquad && \forall w \in \mathring{S}_{r+1}^0(\PST) \text{ for all } T \in \mct, \\
\int_T \dive v\, w & \qquad && \forall w \in \mathring{L}_{r-1}^2(\PST) \text{ for all } T \in \mct.
\end{alignat*}

A function $q \in L_r^2(\mctps)$, if $r \ge 1$, is determined by the degrees of freedom
\begin{subequations}
\begin{alignat*}{2}
q(z) & \qquad &&  1 \leq i \leq 3, \quad  \text{for every vertex } z \text{ of } \mct , \\
q(z) &\qquad && 1 \leq i \leq 3, \quad \forall z \in \mathcal{M}(\mctps), \\
\int_e qp  &\qquad && \forall p \in \pol_{r-2}(e), \ \forall e \in \mathcal{E}^b(\PST), \\
\int_T q & \qquad  && \\
\int_T qp & \qquad && \forall p \in \mathring{L}_r^2(\mctps).
\end{alignat*}
\end{subequations}

Each of the following sequences of spaces forms a complex.
\begin{subequations}\label{seq:global}
\begin{align}
\bbR \lrarrow &S_r^0(\mctps) \overset{\rot}{\lrarrow} L_{r-1}^1(\mctps) \overset{\dive}{\lrarrow} \VV_{r-2}^2(\mctps) \lrarrow 0, \quad r \ge 2,  \label{seq:SLV}  \\
\bbR \lrarrow &S_r^0(\mctps) \overset{\rot}{\lrarrow} S_{r-1}^1(\mctps) \overset{\dive}{\lrarrow} L_{r-2}^2(\mctps) \lrarrow 0, \quad r\ge 3. \label{seq:SSL}
\end{align}
\end{subequations}

\begin{remark}
The spaces  $L_{1}^1(\mctps)$ and $ \dive L_{1}^1(\mctps)$ were considered by Zhang \cite{Zhang08} for approximating incompressible flows. In particular, he proved inf-sup stability of this pair. However, he does not explicitly write the relationship $\VV_{r-2}^2(\mctps) =\dive   L_{r-1}^1(\mctps)$, which we know holds. 
\end{remark}

Additionally, we can define commuting projections. For example, for the sequences \eqref{seq:SLV} and \eqref{seq:SSL}, we  define $\pi_i^r$ such that, for $0 \leq i \leq 2$, $\pi_i^r v|_T = \Pi_i^r (v|_T)$ for all $T \in \mct$. By using Theorem \ref{thm:Commute1}, 
we find that following diagram commutes:
\begin{center}
\begin{tikzcd}
  \mathbb{R} \arrow[r] 
    & C^\infty(S) \arrow[d, "\pi^{r}_0"] \arrow [r,"\rot"] & {[C^\infty(S)]}^2 \arrow[d,"\pi^{r-1}_1"] \arrow[r,"\dive"] & C^\infty(S) \arrow[r] \arrow[d,"\pi^{r-2}_2"]& 0 \\
  \mathbb{R} \arrow[r] & S_{r}^0(\mctps) \arrow[r,"\rot"] & L_{r-1}^1(\mctps) \arrow[r,"\dive"] & \VV_{r-2}^2(\mctps) \arrow[r] & 0.
 \end{tikzcd}
 \end{center}
 Similarly, defining the projections $\chi_i^r v |_T = \varpi_i^r(v|_T)$ for $i = 1,2$, it follows from Theorem \ref{thm:Commute2} that
 the following diagram commutes:
\begin{center}
\begin{tikzcd}
  \mathbb{R} \arrow[r] 
    & C^\infty(S) \arrow[d, "\pi^{r}_0"] \arrow [r,"\rot"] & {[C^\infty(S)]}^2 \arrow[d,"\chi^{r-1}_1"] \arrow[r,"\dive"] & C^\infty(S) \arrow[r] \arrow[d,"\chi^{r-2}_2"]& 0 \\
  \mathbb{R} \arrow[r] & S_{r}^0(\mctps) \arrow[r,"\rot"] & S_{r-1}^1(\mctps) \arrow[r,"\dive"] & L_{r-2}^2(\mctps) \arrow[r] & 0.
 \end{tikzcd}
 \end{center}
The proofs that these projections commute are similar to the local cases.
The top sequences (the non-discrete spaces) are exact if $S$ is simply connected \cite{Costabel}. In the next result, we will show that the bottom sequences (the discrete spaces) are also exact on simply connected domains.

\begin{theorem}
Suppose that $\Omega$ is simply connected. Then the sequence \eqref{seq:SLV} is exact
for $r\ge 2$, and the sequence \eqref{seq:SSL} is exact for $r\ge 3$.
\end{theorem}
\begin{proof}

Suppose that $v\in L_{r-1}^1(\mctps)$ satisfies $\Div v=0$.
Using the inclusion $S_{r-1}^1(\mctps)\subset H({\rm div};\Omega)$
and  standard results, there exists $q\in H({\rm rot};\Omega)$
such that $v = \rot q$.  Because $v$ is a piecewise polynomial
of degree $r-1$, it follows that $q$ is a piecewise polynomial 
of degree $r$.  Moreover, $v$ is continuous and therefore $q\in C^1(S)$.
Thus it follows that $q\in S_r^0(\mctps)$.
Note that this result shows that if $v\in L_{r-1}^1(\mctps)$
satisfies $\Div v=0$, then $v = \rot q$ for some $q\in S_r^0(\mctps)$.

Thus to prove the result, it suffices to show
that the mappings ${\rm div}:L_{r-1}^1(\mctps)\to V_{r-2}^2(\mctps)$
and ${\rm div}:S_{r-1}^1(\mctps)\to L_{r-2}^2(\mctps)$
are surjections.  This will be accomplished by showing
that $\dim ({\rm div}L_{r-1}^1(\mctps)) = \dim V_{r-2}^2(\mctps)$
and  $\dim ({\rm div}S_{r-1}^1(\mctps)) = \dim L_{r-2}^2(\mctps)$.

Denote by $\bbV$, $\bbE$, and $\bbT$
the number of vertices, edges, and triangles
in $\mct$, respectively.
The degrees of freedom given above show that, for $r\ge 2$,
\begin{align*}
\dim S_r^0(\mctps)
%
& = 3\bbV + (4r-8)\bbE+3(r-2)(r-3)\bbT,\\
\dim L_{r-1}^1(\mctps)
%
& = 2\bbV + (4r-6)\bbE + 3(r-2)(r-3)\bbT + \big(3(r-1)r-4\big)\bbT,\\
\dim V_{r-2}^2(\mctps)
& = \bbE+\bbT +\big(3(r-1)r-4\big)\bbT,%
%
\end{align*}
We then find, by the rank--nullity theorem
and the Euler relation $\bbV-\bbE+\bbT=1$ that
\begin{align*}
\dim (\Div L_{r-1}^1(\mctps))
& = \dim L_{r-1}^1(\mctps) - \dim (\rot S_r^0(\mctps))\\
& = \dim L_{r-1}^1(\mctps) - \dim S_r^0(\mctps)+1\\
& = \dim L_{r-1}^1(\mctps) - \dim S_r^0(\mctps)+(\bbV-\bbE+\bbT)\\
& = 2\bbV + (4r-6)\bbE + 3(r-2)(r-3)\bbT + \big(3(r-1)r-4\big)\bbT\\
&\qquad  - \big(3\bbV + (4r-8)\bbE+3(r-2)(r-3)\bbT\big)+(\bbV-\bbE+\bbT)\\
& = \bbE + \bbT+\big(3(r-1)r-4\big)\bbT = \dim V_{r-2}^2(\mctps).
\end{align*}
Likewise, we have for $r\ge 3$,
\begin{align*}
\dim S_{r-1}^1(\mctps)
%
& = 3\bbV+(6r-12)\bbE +  3(r-2)(r-3)\bbT + 3(r-2)(r-3)\bbT,\\
\dim L_{r-2}^2(\mctps)
%
& = \bbV+(2r-5)\bbE + \bbT +3(r-2)(r-3)\bbT,
\end{align*}
and therefore
\begin{align*}
\dim (\Div S_{r-1}^1(\mctps))
& = \dim S_{r-1}^1(\mctps) - \dim S_r^0(\mctps)+(\bbV-\bbE+\bbT)\\
& =  3\bbV +(6r-12)\bbE+6(r-2)(r-3)\bbT\\
&\qquad -\big(3\bbV + (4r-8)\bbE+3(r-2)(r-3)\bbT\big)+(\bbV-\bbE+\bbT)\\
& = \bbV+(2r-5)\bbE +3(r-2)(r-3)\bbT + \bbT = L_{r-2}^2(\mctps).
\end{align*}

\end{proof}

\section{Conclusion}\label{sec:Conclusions}
We have developed smooth finite element spaces on Powell-Sabin splits
that form exact sequences in two dimensions.
We plan to investigate the extension to higher-dimensions in the near future.
Another    interesting question is whether smoother finite element spaces (e.g., $C^2$)
fit an exact sequence on Powell--Sabin triangulations.

\appendix

\section{Proof of Lemma \ref{lem:edge2}}
\begin{proof}
Suppose $z \in W_r(\{a,m,b\})$ is such that \eqref{1d:rgeq1}--\eqref{1d:rgeq2b} are all zero. 
We will show that $z$ must be identically zero on $[a,b]$. Let $\psi(x)$ be a degree $r$ polynomial on the interval $[0,1]$
satisfying
\begin{alignat}{2}
\begin{split}\label{psi}
\psi(0) &= 1,\ \
\psi(1) = 0, \\
\int_0^1 \psi(x) p(x)  &= 0 \hspace{15pt} \forall p \in \pol_{r-2}([0,1]).
\end{split}
\end{alignat}
We note that these conditions uniquely determine $\psi$.
Since $z$ is continuous at $m$ and equal to zero at $a$ and $b$, and in 
view of \eqref{1d:rgeq2a}--\eqref{1d:rgeq2b}, it follows that $z$ may be represented by
\begin{alignat*}{1}
z(y) &= z(m)\begin{cases}
			\psi\left(\frac{y-m}{a-m}\right) & y \in [a,m], \\
			\psi\left(\frac{m-y}{m-b}\right) & y \in [m,b].
			\end{cases}
\end{alignat*}
Since $z'(y)$ is continuous at $m$, it must hold that
\begin{alignat*}{1}
\frac{-1}{m-b} \psi'(0) = \frac{1}{a-m}\psi'(0).
\end{alignat*}
Furthermore, given the conditions \eqref{psi} on $\psi$, we can show that $\psi'(0) \neq 0$. Suppose that $\psi'(0) = 0$ in addition to \eqref{psi}. Then for any $p \in \pol_{r-1}([0,1])$ with $p(0) = 0$, 
\begin{alignat*}{1}
	\int_0^1 \psi'(x) p(x) &= -\int_0^1 \psi(x) p'(x) + \psi(1) p(1) - \psi(0) p(0) = -\int_0^1 \psi(x) p'(x)  = 0
\end{alignat*}
since $p'(x) \in \pol_{r-2}([0,1])$. But $\psi'(x)$ is itself such a function $p(x)$, so it follows that
\begin{alignat*}{1}
	\int_0^1 |\psi'(x)|^2  = 0.
\end{alignat*}
Then $\psi'(x) = 0$, and $\psi$ is constant on $[0,1]$. This contradicts \eqref{psi}, so $\psi'(0) \neq 0$. 
Furthermore, since $1/(b-m) \neq 1/(a-m)$, it follows that $z(m) = 0$. Therefore $z = 0$ on $[a,b]$.
\end{proof}

\section{Proof of Theorem \ref{thm:Commute2}}\label{sec:Commute2}

\begin{proof}
(i) \textit{Proof of \eqref{S0rotS1}.}  Let $p \in C^\infty(T)$ and $\rho := \rot \Pi_0^r p - \varpi_1^{r-1} \rot p \in S_{r-1}^1(\PST)$. 
We show that $\rho$ vanishes on \eqref{S1r}.

First,
\begin{align*}
\rho(z_i) &= \rot \Pi_0^r p(z_i) - \varpi_1^{r-1} \rot p(z_i) = 0,\\
\Div \rho(z_i) & = -\Div \varphi_1^{r-1} \rot p(z_i) = -\Div \rot p(z_i) = 0,
\end{align*}
by the definitions of $\Pi_0^r$ and $\varpi_1^{r-1}$ along with DOFs \eqref{S0r:a} and \eqref{S1r:a}.

Next, if $r=2$,
\begin{align*}
\int_{e_i} \rho\cdot n_i 
& = \int_{e_i} \big(\rot \Pi_0^r p - \varpi_1^{r-1} \rot p \big)\cdot n_i\\
& = \int_{e_i} \big(\rot \Pi_0^r p - \Pi_1^{r-1} \rot p \big)\cdot n_i=0,
\end{align*}
using \eqref{S1r:b}, \eqref{L1r:b} and \eqref{S0rotL1}.  
Similar arguments show that, for $r\ge 3$,
\begin{align*}
\rho(z_{3+i})\cdot n_i 
&= (\rot \Pi_0^r p (z_{3+i}) - \Pi_1^{r-1} \rot p(z_{3+i}))\cdot n_i =0,\\
\int_e \rho\cdot w 
&= \int_e (\rot \Pi_0^r p - \varpi_1^{r-1} \rot p)\cdot w
 = \int_e (\rot \Pi_0^r p - \Pi_1^{r-1} \rot p)\cdot w = 0,
 \intertext{and}
 \int_T \rho \cdot \rot w &= 
 \int_T (\rot \Pi_0^r p - \Pi_1^{r-1} \rot p) \cdot w = 0.
 \end{align*}

Next using \eqref{S1r:c} gives
\begin{align*}
\Div \rho(z_{3+i}) = -\Div \varpi_1^{r-1} \rot p (z_{3+i}) = -\Div \rot p (z_{3+i}) = 0,
\end{align*}
and \eqref{S1r:e} yields
\begin{align*}
\int_e (\Div \rho) q = -\int_e (\Div \varpi_1^{r-1} \rot p) q = -\int_e (\Div \rot p)q = 0
\end{align*}
for all $q\in \pol_{r-4}(e)$ and $e\in \mathcal{E}^b(\PST)$.  The same
arguments, but using \eqref{S1r:g}, gives
\begin{align*}
\int_T (\Div \rho) q=0\qquad \forall q\in \mathring{L}_{r-1}^2(\PST).
\end{align*}
Applying Lemma \ref{lem:S1r} shows
that $\rho\equiv 0$, and so \eqref{S0rotS1} holds.

(ii) \textit{Proof of \eqref{S1divL2}}. For some $v \in [C^\infty(T)]^2$, we define $\rho := \dive \varpi_1^{r-1} v - \varpi_2^{r-2} \dive v \in L_{r-2}^2(\PST)$. Then we need only show that $\rho$ is zero for all DOFs in \eqref{Lr2}. For the vertex DOFs, we have for each $z_i$,
\begin{align*}
\rho(z_i) &= \dive \varpi_1^{r-1} v(z_i) - \varpi_2^{r-2} \dive v(z_i) = 0,
\end{align*}
by \eqref{S1r:a} and \eqref{Lr2:a}. Next, for each $i = 1,2,3$, 
\begin{align*}
\rho(z_{3+i}) &= \dive \varpi_1^{r-1} v(z_{3+i}) - \varpi_2^{r-2} \dive v(z_{3+i}) = 0,
\end{align*}
where we have used \eqref{S1r:a} and \eqref{Lr2:b}.  Similar arguments show
that
\[
\int_e \rho q =0\qquad \forall q\in \pol_{r-4}(e),\ e\in \mathcal{E}^b(\PST),
\]
by \eqref{S1r:e} and \eqref{Lr2:c}, and that
\[
\int_T \rho q = 0\qquad \forall q\in \mathring{L}_{r-2}^2(\PST)
\]
by \eqref{S1r:g} and \eqref{Lr2:d}. Using \eqref{Lr2:c2} and \eqref{S1r:b} if $r = 2$ or \eqref{S1r:d} if $r > 2$,
\begin{align*}
\int_T \rho &= \int_T \dive \varpi_1^{r-1} v  -\varpi_2^{r-2} \dive v = \int_T \dive(\varpi_1^{r-1} v - v) = \int_{\p T} (\varpi_1^{r-1} v - v )\cdot n = 0.
\end{align*}

Therefore, $\rho \equiv 0$ on $T$ by Lemma \ref{lem:Lr2}, and \eqref{S1divL2} is proved.
\end{proof}

\end{document}